\documentclass[11pt,oneside]{amsart}
\usepackage{amsmath,amsfonts,amssymb,latexsym,amsthm}
\usepackage[dvips]{graphics}

\newtheorem{theorem}{Theorem}[section]
\newtheorem{lemma}{Lemma}[section]
\newtheorem{proposition}{Proposition}[section]
\newtheorem{corollary}{Corollary}[section]

\newtheorem{definition}{Definition}[section]

\theoremstyle{remark}

\numberwithin{equation}{section}

\def \Z{\mathbb Z}

\def\X{\mathcal{X}}

\begin{document}
\author{L. Ciobanu, B. Fine, G. Rosenberger}
\date{}
\title{Classes of Groups Generalizing a Theorem of Benjamin Baumslag}

\begin{abstract} In [BB] Benjamin Baumslag proved that being fully residually free is equivalent to being residually free and commutative transitive (CT). Gaglione and Spellman [GS] and Remeslennikov [Re] showed that this is also equivalent to being universally free, that is, having the same universal theory as the class of nonabelian free groups. This result is one of the cornerstones of the proof of the Tarksi problems.  In this paper we extend the class of groups for which Benjamin Baumslag's theorem is true, that is we consider classes of groups $\X$ for which being fully residually $\X$ is equivalent to being residually $\X$ and commutative transitive. We show that the classes of groups for which this is true is quite extensive and includes free products of cyclics not containing the infinite dihedral group, torsion-free hyperbolic groups (done in [KhM]), and one-relator groups with only odd torsion. Further, the class of groups having this property is closed under certain amalgam constructions, including free products and free products with malnormal amalgamated subgroups. We also consider extensions of these classes to classes where the equivalence with universally $\X$ groups is maintained.  
\medskip


\end{abstract}

\maketitle

\section{Introduction} Residual properties have played a major role in infinite group theory. Let $\X$ be a class of groups. Then a group $G$ is {\bf residually} $\X$ if given any nontrivial element $g \in G$ there is a homomorphism $\phi:G \rightarrow H$ where $H$ is a group in $\X$ such that $\phi(g) \ne 1$.  A group $G$ is {\it fully residually} $\X$ if given finitely many nontrivial elements $g_1,...,g_n$
in $G$ there is a homomorphism $\phi:G \rightarrow H$, where $H$ is a group in $\X$, such that
$\phi(g_i) \ne 1$ for all $i=1,...,n$.  Fully residually free groups have played a crucial role in
the study of equations and first order formulas over free groups.  A {\bf universal sentence } in
the language of group theory is a first order sentence using only universal quantifiers (see
[FGMRS]).  The {\bf universal theory} of a group $G$ consists of all universal sentences true in
$G$.  All nonabelian free groups share the same universal theory and a group $G$ is called  {\bf universally free} if it shares the same universal theory as the class of nonabelian free groups. Remeslennikov [Re] and independently Gaglione and Spellman [GS] proved the following remarkable
theorem which became one of the cornerstones in the proof of the Tarski problems (see [Kh 1]and [Se 1].)

\begin{theorem} Suppose $G$ is residually free.  Then the following are equivalent:

$\hphantom{xx}$ (1) $G$ is fully residually free, 

$\hphantom{xx}$ (2) $G$ is commutative transitive, 

$\hphantom{xx}$ (3) $G$ is universally free. 
\end{theorem}

Therefore the class of fully residually free groups coincides with the class of universally free groups. The equivalence of (1) and (2) in the theorem above was proved originally by Benjamin Baumslag ([BB]), where he introduced the concept of fully residually free.

In this paper we consider classes of groups $\X$ for which being fully residually $\X$ is equivalent to being residually $\X$ and commutative transitive, thus extending Baumslag's result.  We prove that the classes of groups for which this is true is quite extensive and includes free groups, torsion-free hyperbolic groups, one-relator groups with only odd torsion, groups acting freely on $\Lambda$-trees where $\Lambda$ is an ordered abelian group, free products of cyclics which do not contain infinite dihedral groups and torsion-free RG-groups (see section 2).  Further the classes of groups with this property is closed under free products and certain amalgamated free products. 

We then consider extending these ideas to classes where the equivalence with universally $\X$ groups is preserved.  Here we need the so-called {\bf big powers} condition. We show that the class where this holds is quite extensive. 

In the next section we introduce some basic definitions and background information and then in section 3 prove our main results.  In section 4 we look at closure properties under amalgam constructions and in section 5 consider universally $\X$ groups.

\section{Background Information} A group $G$ is {\bf commutative transitive} or {\bf CT} if commutativity is transitive on the set of nontrivial elements of $G$. That is if $[x,y] = 1$ and $[y,z] = 1$ for nontrivial elements $x,y,z \in G$ then $[x,z] = 1$.  We want to consider classes of CT groups $\X$ for which the following property holds. We will denote this property by $B\X$. 

\bigskip

\begin{definition} Let $\X$ be a class of groups. If a nonabelian group $G$ is fully residually 
$\X$ if and only if $G$ is residually $\X$ and $CT$, then
we say that $\X$ satisfies $B\X$.
\end{definition}

With this definition B. Baumslag's original theorem says that the class of free groups $\mathcal{F}$ satisfies $B\mathcal{F}$.

\bigskip

In the next section we will prove that a class of groups $\X$ satisfies $B\X$ under very mild conditions and hence the classes of groups for which this is true is quite extensive. However, first we must present a collection of definitions that will play a role in the results. We need the following straightforward characterizations of CT.

\begin{lemma} The following are equivalent:

$\hphantom{xx}$ (1) $G$ is CT,

$\hphantom{xx}$ (2) Centralizers of nontrivial elements of $G$ are abelian.
\end{lemma}

A group $G$ is a {\bf conjugately separated abelian group} or a {\bf CSA group} if maximal abelian subgroups are malnormal. We need the following concerning CSA groups. We give the proofs (see [FGMRS] and the references there).

\begin{lemma} Every CSA group is CT.
\end{lemma} 

\begin{proof} Let $G$ be a CSA group and let $z \in G$ with $z \ne 1$.  We show that the centralizer of $z$ must be abelian and hence $G$ is CT.  Let $a,c$ be in the centralizer of $z$ and let $M_1$ be the maximal abelian subgroup of $G$ containing $a,z$ and $M_2$ the maximal abelian subgroup of $G$ containing $b,z$. We then have $z \in M_1 \cap M_2$.  Let $w \in M_1 \setminus M_2$. Then $w^{-1}xw = z$ is a nontrivial element of $w^{-1}M_2w \cap M_2$ so that $w \in M_2$ a contradiction. Therefore $M_1 \subset M_2$. By maximality then $M_1 = M_2$ and hence the centralizer of $z$ is abelian.  It follows that $G$ is CT. 

\end{proof}

The converse is not true (see [FMgrRR ]). However, as we will show, in the presence of property $B\X$ CSA is equivalent to CT.

\begin{lemma} Let $G$ be a CSA group and let $H$ be a subgroup of $G$.  Then $H$ is also a CSA group.
\end{lemma}

\begin{proof} (see [GKM]) Let $G$ be a CSA group and let $H$ be a subgroup of $G$. Let $A_H$ be a maximal abelian subgroup of $H$. We must show that $A_H$ is malnormal in $H$. Let $x \in H$ with $xA_Hx^{-1} \cap A_H \ne \{1\}$. $A_H$ is contained in a maximal abelian subgroup $A_G$ of $G$. Since $G$ is CSA it follows that $A_G$ is malnormal in $G$ and so $x \in A_G$.  Then $x \in (A_G \cap H) \subset A_H$ and hence $A_H$ is malnormal in $H$.
\end{proof}

Recall that the infinite dihedral group has the presentation $D = <x,y; x^2 = y^2 = 1>$.  Then $xxyx^{-1} = yxyy^{-1} = yx = (xy)^{-1}$ and hence $D$ is not CSA.  We need this in considering classes satisfying $B\X$.

\begin{lemma} If $G$ is a group that contains the infinite dihedral group $D$ then $G$ is not CSA.
\end{lemma}

Beyond CSA we need the following ideas.  A group $G$ is {\bf power commutative} if $[x,y^n] =1$ implies that $[x,y] = 1$ whenever $y^n \ne 1$.  Two elements $a,b \in G$ are in {\bf power relation} to each other if there exists an $x \in G\setminus \{1\}$ with $a = x^n,b=x^m$ for some $n,m \in \Z$.  $G$ is {\bf power transitive} or {\bf PT} if this relation is transitive on nontrivial elements.  

A group $G$ is an {\bf RG- group} or {\bf Restricted-Gromov group} if for any $g,h \in G$ either the subgroup $<g,h>$ is cyclic or there exists a positive integer $t$ with $g^t \ne 1, h^t \ne 1$ and $<g^t,h^t> = <g^t> \star <h^t>$.  Note that torsion-free hyperbolic groups are RG-groups.  

The following ideas are crucial in handling property $B\X$.

A group $G$ is {\bf ALC} if every abelian subgroup is locally cyclic.  This is of course the case in free groups. A group $G$ is {\bf ANC} if every abelian normal subgroup is contained in the center of $G$. Finally a group is {\bf NID} if it does not contain a copy of the infinite dihedral group $\Z_2 \star \Z_2$.  If $G$ has only odd torsion then clearly $G$ is NID.

\section{Classes of groups $\X$ Satisfying $B\X$}

We now prove the following.

\begin{theorem} Let $\X$ be a class of groups such that each nonabelian $H \in \X$ is CSA. Let $G$ be a nonabelian and residually $\X$ group.  Then the following are equivalent

$\hphantom{xx}$ (1) $G$ is fully residually $\X$ 

$\hphantom{xx}$ (2) $G$ is CSA

$\hphantom{xx}$ (3) $G$ is CT

Therefore the class $\X$ has the property $B\X$.
\end{theorem}

To prove this we first need the following lemmas.

\begin{lemma} CSA implies ANC, that is, if $G$ is a CSA group then $G$ is ANC.  Hence if a class of groups satisfies CSA then it also satisfies ANC.
\end{lemma}

\begin{proof} Let $G$ be a CSA group and let $A$ be an abelian normal subgroup of $G$.  Then $A$ is contained in a maximal abelian subgroup $B$ of $G$.  Let $a \in A$ and $x \in G$.  Then $xax^{-1} \in A$ since $A$ is normal. Now $a \in B$ and hence $xax^{-1} \in B$. Since $G$ is CSA it follows that $B$ is malnormal and therefore $x \in B$. Since $B$ is abelian it follows that $xax^{-1} = a$ for all $x \in G$ and therefore $A$ is in the center of $G$. Hence $G$ satisfies ANC.

\end{proof}

\begin{lemma} Let $\X$ be as in Theorem 3.1 and let $G$ be nonabelian and residually $\X$. Let $A$ be an abelian normal subgroup of $G$. Then $A$ is contained in the center of $G$. In particular if $G$ is CT then $A$ must be trivial.
\end{lemma}

\begin{proof} First notice that if $G \in \X$ then by assumption $G$ is CSA and therefore $G$ is ANC from Lemma 3.1. Now let $G$ be nonabelian and residually $\X$. Let $A$ be an abelian normal subgroup of $G$. Suppose that $A$ is not contained in the center of $G$.  Then there exist $a \in A$ and $b \in G$ such that $[a,b] \ne 1$.  Since $G$ is nonabelian and residually $\X$ there exists a normal subgroup $N$ of $G$ with $G/N \in \X$ and $[a,b] \ne 1$ modulo $N$.  However $AN/N$ is a nontrivial normal abelian subgroup of $G/N$ and since $G/N$ is ANC as a group in $\X$, $AN/N$ must be contained in the center of $G/N$. But this contradicts $[a,b] \ne 1$ modulo $N$. Therefore $A$ is contained in the center of $G$.

If $G$ is also CT and nonabelian then it is centerless so from the above it follows that any normal abelian subgroup must be trivial.
\end{proof}

Now we give the proof of Theorem 3.1.

\begin{proof} We first show that (1) implies (2).  Suppose that $G$ is fully residually $\X$.  Since $G$ is nonabelian it follows that $G$ has nontrivial elements.  Let $u$ be a nontrivial element of $G$ and  denote by $M$ the centralizer $C_G(u):$
$$C_G(u) = \{ x \in G ; ux = xu \} = M.$$

Then $M$ is maximal abelian in $G$. We claim that $M$ is malnormal in $G$.  Suppose that $w = g^{-1}ug \ne 1$ lies in $g^{-1}Mg \cap M$. If $g \notin M$ then $[g,u] \ne 1$. Then there is a group $H$ in $\X$ and an epimorphism $\phi:G \rightarrow H$ taking $x \rightarrow \overline{x}$ such that $\overline{w} \ne 1$ and $[\overline{g},\overline{u}] \ne 1$.  Let $C = C_H(\overline{u})$. Then $\overline{w} \in (\overline{g}^{-1}C\overline{g}) \cap C$.

However $H \in \X$ and hence by assumption $H$ is CSA so that the maximal abelian subgroups of $H$ are malnormal. This implies that $\overline{g} \in C$ contradicting $[\overline{g},\overline{u}] \ne 1$. This contradiction shows that $g^{-1}Mg \cap M \ne 1$ implies that $g \in M$. Hence the maximal abelian subgroups of $G$ are malnormal and therefore $G$ is CSA. This proves that (1) $\implies$ (2).

We now show that (2) implies (3), that is that CSA implies CT. This follows from Lemma 3.1.  

Finally we show that (3) implies (1).  Let $G$ be a CT, nonabelian and residually $\X$ group. We show that $G$ is fully residually $\X$. 

Let $g_1,g_2,...,g_n$ be a set of nontrivial elements of $G$. We must show that there is a group $H \in \X$ and an epimorphism $\phi:G \rightarrow H$ such that $\phi(g_i) \ne 1$ for all $i = 1,...,n$. Recall this is equivalent to showing that given $g_1,...,g_n$ nontrivial elements in $G$ there is a normal subgroup $N$ with $g_1,...,g_n$ not in $N$ and $G/N \in \X$.  We show this by induction on the size $n$ of the set of elements.  

Since $G$ is assumed to be residually $\X$ this is true for $n = 1$.  

Assume that for $n \ge 1$ given nontrivial elements $g_1,..,g_n \in G$ there exists a nontrivial $g \in G$ such that for all normal subgroups $N$ of $G$ if $g \notin N$ then $g_i \notin N$ for $i = 1,...,n$.  Since $G$ is residually $\X$ this is true for $n = 1$.  We show that given $g_1,...,g_n,g_{n+1}$ we can find a $g' \ne 1$ such that if $g' \notin N$ for any normal subgroup $N$ of $G$ then $g_i \notin N$ for $i = 1,...,n+1$.  Let $g$ be the assumed element for $g_1,...,g_n$ and for each $x \in G$ let $c(x) = [g,xg_{n+1}g^{-1}]$.  If $c(x) = 1$ for all $x$ then each conjugate of $g_{n+1}$ commutes with $g \ne 1$.  Then by commutative transitivity the normal closure $N_{g_{n+1}}$ is abelian and hence here trivial from Lemma 3.2.  But $g_{n+1}$ is contained in it and nontrivial. Therefore $c(x) \ne 1$ for some $x \in G$.  Choose this nontrivial $c(x)$ as $g'$. Then if $g' \notin N$ for a normal subgroup $N$ of $G$ it follows that $g_1,...,g_{n+1} \notin N$.  It follows from this induction that for each finite set $g_1,...,g_n \in G$ there is a $g \in G$ such that if $g \notin N$ for some normal subgroup $N$ of $G$ then $g_1,...,g_n$ is also not in $N$.  Since $G$ is residually $\X$ it follows that there is such an $N$ with $G/N \in \X$. Therefore $G$ is fully residually $\X$ proving that (3) implies (1) and completing the proof of Theorem 3.1. 

\end{proof}

Hence a class of groups $\X$ satisfies $B\X$ if each nonabelian $H \in \X$ is CSA. Examples of $B\X$ classes abound. In particular we list the following.

\begin{theorem} Each of the following classes satisfies $B\X$:

$\hphantom{xx}$ (1) The class of nonabelian free groups.

$\hphantom{xx}$ (2) The class of noncyclic torsion-free hyperbolic groups (see [FR]).

$\hphantom{xx}$ (3) The class of noncyclic one-relator groups with only odd torsion (see [FR]).

$\hphantom{xx}$ (4) The class of cocompact Fuchsian groups with only odd torsion.

$\hphantom{xx}$ (5) The class of noncyclic groups acting freely on $\Lambda$-trees where $\Lambda$ is an ordered abelian group (see [H]).

$\hphantom{xx}$ (6) The class of noncylic free products of cyclics with only odd torsion.

$\hphantom{xx}$ (7) The class of noncyclic torsion-free RG-groups (see[FMgrRR] and [AgrRR]).

$\hphantom{xx}$ (8) The class of conjugacy pinched one-relator groups of the following form
$$G = <F,t; tut^{-1} = v>$$
where $F$ is a free group of rank $n \ge 1$ and $u,v$ are nontrivial elements of $F$ that are not proper powers in $F$ and for which $<u> \cap x<v>x^{-1} = \{1\}$ for all $x \in F$.

\end{theorem}

The theorem follows from the fact that each of these classes has the property that each nonabelian group in them is CSA.

\smallskip

Since CSA always implies CT we have the following corollary.

\begin{corollary} Let $\X$ be a class of CSA groups. Then if $G$ is a nonabelian residually $\X$ group then CT is equivalent to CSA.
\end{corollary}   

Commutative transitivity (CT) has been shown to be equivalent to many other properties (see[AgrRR]) under the additional condition that abelian subgroups are locally cyclic (ALC) Hence we get the corollary.

\begin{corollary} Let $\X$ be a class of groups such that each nonabelian $H \in \X$ is CSA. Let $\mathcal{Y}$ be the subclass of $\X$ consisting of those groups in $\X$ which are ALC. Let $G$ be a nonabelian residually $\mathcal{Y}$ group which is ALC and has trivial center.  Then the following are equivalent.

$\hphantom{xx}$ (1) $G$ is fully residually $\mathcal{Y}$. 

$\hphantom{xx}$ (2) $G$ is CSA.

$\hphantom{xx}$ (3) $G$ is CT.

$\hphantom{xx}$ (4) $G$ is PC.

$\hphantom{xx}$ (5) $G$ is PT.
\end{corollary}

This follows directly from the equivalences given in [AgrRR].

\section{Some Results on Closure} Here we show that classes of groups satisfying $B\X$ are closed under certain amalgam constructions.  This follows from the fact that CT and hence CSA is preserved under such constructions.

\begin{proposition} Let $\X$ be a class of CSA groups closed under free products with malnormal amalgamated subgroups, in particular under free products. Let $G_1$ and $G_2$ be nonabelian, residually $\X$ and CSA. Then $G_1 \star G_2$ and $G_1 \star_A G_2$ where $A$ is malnormal in $G_1$ and $G_2$ are both residually $\X$ and CSA.
\end{proposition}

\begin{proof} We show it for free products, as the proof for free products with malnormal amalgamation is analogous. Let $G_1$ and $G_2$ be nonabelian, residually $\X$ and CSA.  We show that $G_1 \star G_2$ is also residually $\X$ and CSA.

From Theorem 3.1 since $G_1$ and $G_2$ are nonabelian, residually $\X$ and CSA they are both fully residually $\X$ and CT. Let $g \in G_1 \star G_2$ with $g \notin G_1$ and $g \notin G_2$..  Then, up to conjugacy,
$$g = a_1b_1a_2b_2 \cdots a_nb_n$$
where $a_1,...,a_n,b_1,...,b_n$, the syllables of $g$, are nontrivial elements of $G_1,G_2$ respectively.  Then there exist
maps $\phi_{G_1}:G_1 \rightarrow H_1$ and $\phi_{G_2}:G_2 \rightarrow H_2$
with $H_1,H_2$ in $\X$ which do not annihilate any of the syllables
in $g$.  This can then be extended to a map $\phi_g:G \rightarrow
H_1 \star H_2$ which does not annihilate $g$. Since $\X$ is closed under free products this shows that $G$ is
residually $\X$. Further a free product of CT groups is CT (see [LR]) hence $G_1 \star G_2$ is also CT.  Since it is residually $\X$ and CT it is CSA and fully residually $\X$.

\end{proof}

The following then follow easily from this proposition.

\begin{corollary} Let $\X$ be a class of nonabelian CSA groups. Hence $\X$ satisfies $B\X$. Let $G_1$ and $G_2$ be fully residually $\X$ groups and $G = G_1 \star G_2$.  Then $G$ is fully residually $\X$.
\end{corollary}

\begin{corollary} Let $\X$ be a class of nonabelian CSA groups. Hence $\X$ satisfies $B\X$. Consider the class of groups which are free products $G_1 \star G_2$ of groups from $\X$. Then this class also satisfies $B\X$.
\end{corollary}

\begin{corollary} Let $\X$ be a class of nonabelian CSA groups. Hence $\X$ satisfies $B\X$.Consider the class of groups which have the form $G_1 \star_A G_2$ where $G_1,G_2$ are groups from $\X$ and $A$ is malnormal in $G_1$ and $G_2$. Then this class also satisfies $B\X$.
\end{corollary}

\section{Universally $\X$-groups and the Big Powers Condition}

We now consider the equivalence with universally $\X$ groups. We say that a group $G$ is {\bf universally} $\X$ if it satisfies the universal theory of a countable nonabelian group from $\X$. Recall that if $\X$ is the class of free groups we have the following equivalence mentioned in the introduction.

\begin{theorem} Suppose $G$ is residually free.  Then the following are equivalent:

$\hphantom{xx}$ (1) $G$ is fully residually free, 

$\hphantom{xx}$ (2) $G$ is commutative transitive, 

$\hphantom{xx}$ (3) $G$ is universally free. 
\end{theorem}

Our Theorem 3.1 shows the equivalence of (1) and (2) for any class $\X$ of CSA groups.    
To prove an equivalence with (3) we need the {\bf big powers condition}. This was introduced originally by G.Baumslag in [GB].

\begin{definition} Let $G$ be a group and $u = (u_1,...,u_k)$ be a sequence of nontrivial elements of $G$. Then 

$\hphantom{xx}$ (1) $u$ is {\bf generic} if neighboring elements in $u$ do not commute, that is $[u_1,u_{i+1}] \ne 1$ for every $i \in \{1,...,k\}$.  

$\hphantom{xx}$ (2) $u$ is {\bf independent} if there exists an $n = n(u) \in \Bbb{N}$ such that for any $\alpha_1,...,\alpha_k \ge n$ we have $u_1^{\alpha_1} \cdots u_k^{\alpha_k} \ne 1$. 

$\hphantom{xx}$ (3) A group satisfies the {\bf big powers condition} or {\bf BP} if every generic sequence in $G$ is independent.  We call such groups $BP$-groups.
\end{definition}

 G. Baumlsag proved that free groups are BP-groups [GB] while Olshansky [O] showed that torsion-free hyperbolic groups are BP-groups. For BP groups the following results are known.

\begin{lemma} [KMS] A subgroup of a BP-group is itself a BP-group.
\end{lemma}

\begin{lemma} [O] Every torsion-free hyperbolic group is a BP-group
\end{lemma}

A stronger version of this lemma for relatively hyperbolic groups is given in [KM].

\begin{lemma} A free product of CSA BP-groups is also a BP-group
\end{lemma}

\begin{lemma} Let $G = F_1 \underset{u=v}{\star} F_2$ where $F_1,F_2$ are finitely generated free groups and $u,v$ are nontrivial elements of $F_1,F_2$ respectively with not both proper powers. Then $G$ is a CSA BP-group.
\end{lemma}

We now consider a class of groups $\mathcal{Z}$ in which each finitely presented nonabelian group $G$ in $\mathcal{Z}$ is CSA and BP.  Reinterpreting a result in [BMR 1] and [BMR 2] (see also [KM]) we obtain the following using the same proof utilizing the BP condition.

\begin{theorem} Let $\mathcal{Z}$ be a class of finitely presented groups such that each nonabelian $H \in \mathcal{Z}$ is CSA and BP.  Let $H \in \mathcal{Z}$ and $G$ a nonabelian group. Then the following are equivalent.

$\hphantom{xx}$ (1) $G$ is fully residually $H$,

$\hphantom{xx}$ (2) $G$ is universally equivalent to $H$.

\end{theorem}

Finally combining our results we get

\begin{theorem} Let $\mathcal{Z}$ be a class of finitely presented groups such that each nonabelian  $H \in \mathcal{Z}$ is CSA and BP. Let $G$ be a nonabelain residually $\mathcal{Z}$ group. Then the following are equivalent

$\hphantom{xx}$ (1) $G$ is fully residually $\mathcal{Z}$,

$\hphantom{xx}$ (2) $G$ is CSA,

$\hphantom{xx}$ (3) $G$ is CT,

$\hphantom{xx}$ (4) $G$ is universally $\mathcal{Z}$.

\end{theorem}

Equivalences (1),(2),(3) are from Theorem 3.1 and the fact that $\mathcal{Z}$ consists of CSA groups while the equivalence with (4) follows from the BP condition. As before, in the case of ALC groups there are additional equivalences.

\begin{corollary} Let $\mathcal{C}$ be a class of finitely presented groups such that each nonabelian group in $\mathcal{C}$ is CSA, ALC and BP. Let $G$ be a nonabelian residually $\mathcal{C}$ group which is ALCand has trivial center. Then the following are equivalent.

$\hphantom{xx}$ (1) $G$ is fully residually $\mathcal{C}$,

$\hphantom{xx}$ (2) $G$ is CSA,

$\hphantom{xx}$ (3) $G$ is CT,

$\hphantom{xx}$ (4) $G$ is PC,

$\hphantom{xx}$ (5) $G$ is PT,

$\hphantom{xx}$ (6) $G$ is universally $\mathcal{C}$.

\end{corollary}

As with classes of CSA groups, the subclasses of these which are also BP are quite extensive.

\begin{theorem} The following classes of groups consist of groups that are both CSA and BP. Hence the equivalences in Theorem 5.2 hold for any group $G$ which is residually in any of these classes.

$\hphantom{xx}$ (1) The class of nonabelian finitely generated free groups,

$\hphantom{xx}$ (2) The class of noncyclic torsion-free hyperbolic groups,

$\hphantom{xx}$ (3) The class of noncyclic cyclically pinched one-relator groups $F_1 \underset {u=v} {\star} F_2$ with not both $u,v$ proper powers in their respective finitely generated free groups,

$\hphantom{xx}$ (4) The class of free products $G_1 \star G_2$ where both $G_1$ and $G_2$ are nonabelian, finitely presented and satisfy CSA and BP.

\end{theorem}

 \section*{Acknowledgments}
%%%%%%%%%%%%%%%%%%%%%%%%%%%%%%%%%%%%%%%%%%%%%%%%%%%%%%%%%%%%%%%%%%%%%%%%%%%%
%%%%%%%%%%%%%%%%%%%%%%%%%%%%%%%%%%%%%%%%%%%%%%%%%%%%%%%%%%%%%%%%%%%%%%%%%%%%

The authors were partially supported by the Marie Curie Reintegration Grant 230889.

 \section{References}

\noindent \lbrack AgrRR] P. Ackermann, V. gr.Rebel, G. Rosenberger, On Power and Commutation Transitive, Power Commutative and Restricted Gromov Groups,  \textbf{Cont. Math}, 360, 2004,
1--4.

\noindent \lbrack BB] B.Baumslag, Residually free groups, \textbf{\ Proc. London Math. Soc.} (3), 17,1967, 635 -- 645.

\noindent \lbrack GB] G. Baumslag,  On generalized free products, \textbf{ Math. Z.},  78, 1962, 423-438.

\noindent \lbrack BMR 1] G. Baumslag, A.Myasnikov and V.Remeslennikov, Discriminating completions of hyperbolic groups, \textbf{Geometrica Dedicata},  92, 2002, 115-143.

\noindent \lbrack BMR 2] G. Baumslag, A.Myasnikov and V.Remeslennikov, Algebraic geometry over groups I. Algebraic sets and ideal theory, \textbf{J.Algebra}, 219, 1999, 16-79.

\noindent \lbrack FR] B. Fine and G. Rosenberger, On Restricted Gromov Groups \textbf{Comm.in Alg.}, 20 , 1992, 2171-2182.

\noindent \lbrack FMgrRR] B. Fine, A.Myasnikov, V. gr. Rebel and G. Rosenberger, A Classification of Conjugately Separated Abelian, Commutative Transitive and Restricted Gromov One-Relator Groups, \textbf{Result. Math.}, 50, 2007, 183-193.

\noindent \lbrack FGMRS] B.Fine, A. Gaglione, A. Myasnikov,
G.Rosenberger and D. Spellman,  A Classification of Fully
Residually Free Groups of Rank Three or Less, \textbf{  J. of Algebra
}, 200,  1998,  571--605. 

\noindent \lbrack GS] A. Gaglione and D. Spellman,  Some Model Theory of Free Groups and Free
Algebras, \textbf{ Houston J. Math }, 19,  1993,  327-356.

\noindent \lbrack GKM] D.Gildenhuys, O. Kharlampovich and A. Myasnikov,  CSA Groups and Separated
Free Constructions, \textbf{ Bull. Austral. Math. Soc.}, 52,  1995,  63-84. 

\noindent \lbrack H] N.Harrison,  Real Length Functions in Groups, \textbf{Trans. Amer. Math. Soc.}, 174,  1972, 77-106.

\noindent \lbrack  KhM 1] O. Kharlamapovich and A.Myasnikov, 
Irreducible affine varieties over a free group: I. Irreducibility
of quadratic equations and Nullstellensatz, \textbf{  J. of Algebra
}, 200,  1998,  472-516.

\noindent \lbrack KMS] A.V. Kvaschuk, A.G. Myasnikov, D.E. Serbin, Pregroups with the Big Powers Condition, \textbf{ Algebra and Logic} 48. 2009, 193-213.

\noindent \lbrack KM] A.V. Kvaschuk, A.G. Myasnikov, On the Big Powers Condition, preprint.

\noindent \lbrack LR] F. Levin, G. Rosenberger, On Power-Commutative and Commutation-Transitive Groups,
\textbf{London Math. Soc. Lect. Note 121}, Cambridge University Press (1986), 249-253.

\noindent \lbrack O] A. Olshanskii, On Relatively Hyperbolic and G-subgroups of Hyperbolic Groups \textbf{Int. J. Alg. and Compt} 3, 1993, 365-409.

\noindent \lbrack Re] V.N. Remeslennikov,   $\exists$-free groups, \textbf{ Siberian Mat. J. }, 30, 1989,  998--1001. 

\noindent \lbrack Se] Z. Sela,  Diophantine Geometry over Groups V: Quantifier Elimination, \textbf{Israel Jour. of Math.}, 150,  2005,  1-97.

\smallskip
\noindent Laura Ciobanu, University of Fribourg, Department of Mathematics, Chemin du Mus\'ee 23, 1700 Fribourg, Switzerland.\\
 \textit{E-mail}: \textrm{laura.ciobanu@unifr.ch}\\
 
 \noindent Benjamin Fine, Fairfield University, Fairfield, CT 06430, USA\\
 \textit{E-mail}: \textrm{fine@fairfield.edu}\\
 
 \noindent Gerhard Rosenberger,Fachbereich Mathematik, University of Hamburg,Bundestrasse 55, 20146 Hamburg, Germany\\
 \textit{E-mail}: \textrm{Gerhard.Rosenberger@math.uni-hamburg.de}

\end{document}